\documentclass[12pt]{amsart}
\usepackage{graphicx}
\usepackage[all]{xy}
\usepackage{amscd}
\usepackage{amssymb}
\usepackage{amsmath}
\usepackage{amsthm}
\usepackage{amsfonts}
\usepackage{graphics}
\usepackage{epsfig,psfrag}
\usepackage[english]{babel}
\usepackage{latexsym}
\usepackage{mathrsfs}
\usepackage{color}
\usepackage{url}

\usepackage{hyperref}

\newtheorem{theorem}{Theorem}[section]
\newtheorem{corollary}[theorem]{Corollary}
\newtheorem{lemma}[theorem]{Lemma}

\newtheorem{proposition}[theorem]{Proposition}
\newtheorem{question}[theorem]{Question}

\theoremstyle{definition}

\newtheorem{definition}[theorem]{Definition}

\newcommand{\ZZ}{{\mathbb Z}}

\theoremstyle{remark}
\newtheorem{remark}[theorem]{Remark}

\begin{document}
\title[ Degree one maps on four manifolds with cyclic $\pi_1$]
{Degree one maps on four manifolds with cyclic fundamental groups}

\author{Yang Su}
\address{School of Mathematics and Systems Science, Chinese Academy of Sciences, Beijing 100190, China}
\address{School of Mathematical Sciences, University of Chinese Academy of Sciences, Beijing 100049, China}
\email{suyang@math.ac.cn}

\author{Shicheng Wang}
\address{School of Mathematical Sciences, Peking University, Beijing, 100871, CHINA}
\email{wangsc@math.pku.edu.cn}

\author{Zhongzi Wang}
\address{School of Mathematical Sciences, Peking University, Beijing, 100871, CHINA}
\email{wangzz22@stu.pku.edu.cn}

\begin{abstract}
We study degree one maps between closed orientable 4-manifolds with cyclic $\pi_1$, and obtain some results on the existence and finiteness, as well as  some relation of $1$-domination and Euler Characteristics.
\end{abstract}  
\date{}
\maketitle
\tableofcontents
\section{Introduction}

\subsection{Description of main results} 
For a finite CW-complex $X$, we use $H_i(X)$, $H^i(X)$, $\beta_i(X)$,
$\pi_i(X)$, $\chi(X)$ to denote its $i$-th  integral homology group, $i$-th integral cohomology group,  $i$-th Betti number, $i$-th homotopy group, and 
Euler Characteristic respectively. For  a closed oriented 4-manifold $X$, we use  $I_X$, $\sigma(X)$, $w_2(X)$ and $I(X, \ZZ_2)$, and $ks(X)$ to denote its integral intersection form,   signature,  second Stiefel-Whitney class,  $\ZZ_2$-intersection form,
and Kirby-Siebenmann invariant respectively. 
For a finitely presented group $G$, we use $G^{ab}$ to denote its the abelianization,  and $\beta_1(G)$ to denote the rank of the free part
of $G^{ab}$.
For  terminologies in 4-manifolds appeared above and after, see \cite{Kir}.

\begin{definition}  Suppose $X$ and $Y$ are closed oriented $n$-manifolds.
Call $X$ {\it 1-dominates} $Y$ if there exists a  map  $f: X\to Y$ of  degree one.
When $n=4$,  call $X$ stably 1-dominates $Y$, if  $X\# W$
1-dominates $Y\# W$, where $W=\#_{k}S^2\times S^2$ for some $k$.
We will use ``degree one map" and ``1-domination" both in the coming context.
\end{definition}

Degree one maps  between closed orientable $n$-manifolds is an old and interesting topic.
Results on this topic usually relies on the understanding of the manifolds under discussion.
For examples, when there are degree one  maps between closed orientable 2-manifolds is known and is based on the classification of closed orientable 2-manifolds.
With Thurston's picture on 3-manifolds, many results on finiteness, existence, and rigidity  about degree one maps between closed  orientable 3-manifolds  
appeared in last several decades,  see  
\cite{BRW}, \cite{Liu} and the references therein.
With the classification of simply connected 4-manifolds due to Freedman \cite{Fr}, some results on the existence and finiteness about 1-domination
on simply connected 4-manifolds follow quickly \cite{DW}: 
\begin{enumerate}
\item For closed simply connected 4-manifolds $X$ and $Y$, $X$ 1-dominates  $Y$ if and only if their intersection forms
satisfy  $I_X=I_Y \oplus H$. 

\item Every closed simply connected 4-manifold 1-dominates at most finitely many closed orientable 4-manifolds.
\end{enumerate}

On the other hand, after Freedman's work \cite{Fr} there are classification results of $4$-manifolds with non-trivial fundamental groups, see \cite{FQ}, \cite{HK1}, \cite{HK2},
\cite{KLPT} and the references therein. Based on these results,
in this note we systematically study degree one maps between  closed orientable 4-manifolds with cyclic fundamental groups, which is the best  understood class of 4-manifolds after the simply connected case. The remaining part of this subsection is a summary of the results in this note. 

We first prove some results analogous to  (1)  and (2) above:

(1)  Suppose  $X$ and $Y$ are closed orientable 4-manifolds  with $\pi_1(X)=\pi_1(Y)$, $I_X=I_Y \oplus H$ and $I(X, \ZZ_2)=I(Y, \ZZ_2)\oplus L$.
 If $\pi_1=\mathbb{Z}$,
 then $X$ stably 1-dominates $Y$, but not necessarily 1-dominates $Y$ (Theorems  \ref{stably 1-dom} and \ref{non 1-dom});
 If $\pi_1=\mathbb{Z}_n$, then  $X$ 1-dominates $Y$ unless the following case:
 $n$  is even, $X$ is almost spin, and $Y$ is spin  (Theorem \ref{1-dom n}).
 
(2) For any given $n\ge 0$, there exists a closed orientable 4-manifold  that 1-dominates all closed orientable 4-manifolds $Y$ with $\pi_1(Y)=\ZZ$ and $\beta_2(Y) \le n$, and every closed orientable 4-manifold $X$  1-dominates at most finitely many stable homeomorphism classes of closed orientable 4-manifolds $Y$ with $\pi_1(Y)=\mathbb{Z}$ (Theorems  \ref{infinite cyclic} and \ref{cyclic-stable}).
Every closed orientable 4-manifold $X$  1-dominates at most finitely many closed orientable 4-manifolds $Y$ with finite abelian $\pi_1$ (Theorem \ref{finite-abelian}).

We also address the relations between 1-domination and Euler characteristic.
Recall  that for a finitely presented group $G$,  its \emph{$4$-manifold Euler characteristic}  is defined by  \cite{HW}
$$\chi_4(G)=\text{inf}\{\chi(X)| \text{$X$ is a closed orientable 4-manifold and $\pi_1(X)\cong G$}\}.$$ 
This invariant of groups is  extensively studied  when $G\cong \pi_1(Q)$  for compact 3-manifolds $Q$, \cite{Ko}, \cite{Hi1}, \cite{KL}, \cite{SunW1}. 
Call a closed orientable 4-manifold $X$ realizing $G$ if $\pi_1(X) \cong G$ and $\chi(X)=\chi_4(G)$. Both the Euler characteristic and the relation of 1-domination  measure the complexity of 4-manifolds with a given fundamental group.
People asked \cite{SunW2} if $N$  is a closed orientable 4-manifold
with $\pi_1(N)\cong G$,  does $N$ 1-dominate a 4-manifold $X$ realizing $\chi_4(G)$? 
We will give a positive answer to this question when $G$ is  cyclic (Theorem \ref{answer}).  

The proofs of above results are quite different for $\pi_1=\ZZ$ and $\pi_1=\ZZ _n$, which  rely on the classification results of Freedman-Quinn \cite{FQ} and a decomposition theorem of Hambleton-Teichner \cite{HT} for $\pi_1=\ZZ$,  
and the classification results of  Hambleton-Kreck \cite{HK1}, \cite{HK2} and a decomposition theorem (Theorem \ref{decomposition}) derived from the classification results for $\pi_1=\ZZ _n$.


\subsection{Some elementary facts and classification results}

First we state a primary  fact for algebraic topology on 4-manifolds.
 
\begin{lemma}\label{Euler}  Suppose $X$ is a closed oriented 4-manifold.

(1) Poincare duality: $\beta_1(X)=\beta_3(X)$.

(2) Euler-Poincare formula: $\chi(X)=2-2\beta_1(X)+\beta_2(X).$
\end{lemma}

Then state a couple of known facts about 1-dominations.

 \begin{lemma}\label{pinch} Suppose $X, Y, X', Y'$ are closed oriented n-manifolds.
 
(1) The connected sum  $X\# Y$ 1-dominates  $X$.

(2) If $X$ 1-dominates $Y$, $X'$ 1-dominates $Y'$, then $X\# Y$ 1-dominates $X'\# Y'$.
\end{lemma}\label{splitting}

\begin{lemma}\label{splitting}
Suppose $X, Y,$ are closed oriented 4-manifolds. If $X$ 1-dominates $Y$, 
then  there are direct sum decompositions

(i) $H_*(X, F)=H_*(Y, F)\oplus G$, for $F=\ZZ$ or $\ZZ_2$ 

and orthogonal sum decompositions

(ii) $I_X = I_Y \oplus L \ \ \ \text{and}\,\, I(X, \ZZ_2)=I(Y, \ZZ_2)\oplus L'\ \ (\star).$
\end{lemma}

\begin{remark}\label{no 2-torsion}

By the universal coefficient theorem
$$ 0 \to \mathrm{Ext}(H_1(X), \ZZ_2) \to H^2(X;\ZZ_2) \to \mathrm{Hom}(H_2(X), \ZZ_2) \to 0,$$
we see that if $H_1(X)$ has no 2-torsion, then the information of $I(X, \ZZ_2)$ is covered by that of $I_X$.

If $X$ 1-dominates $Y$ and $H_1(X)$ has no 2-torsion, then $H_1(Y)$ has no 2-torsion. Then the condition $(\star)$
in Lemma \ref{splitting}  (ii) can simply be presented as  $I_X = I_Y \oplus L.$
\end{remark}

Now we state some classification results for closed orientable 4-manifold with cyclic $\pi_1$.

Let $X$ be a closed orientable 4-manifold with $\pi_1(X)=\mathbb{Z}$. 
Let $\Lambda=\mathbb{Z}[\mathbb{Z}]=\mathbb{Z}[t,t^{-1}]$ be the group ring of the infinite cyclic group $\mathbb{Z}$ with generator $t$.  Let $\widetilde X$ be the universal cover of $X$, then the equivariant intersection form on the finitely-generated free $\Lambda$-module $\pi_2(X)=\pi_2(\widetilde X)=H_2(\widetilde X)$
$$S(\widetilde X) \colon \pi_2(X) \times \pi_2(X) \to \Lambda.$$
is a non-singular hermitian form over $\Lambda$ (c.~f.~\cite[\S 3]{H}).

Now we state the  classification result of Freedman and Quinn.
\begin{theorem}\label{F-Q}\cite{FQ}
A closed, oriented topological $4$-manifold $X$ with $\pi_1(X) = \mathbb Z$ is classified up to homeomorphism by $S(\widetilde X)$ and $ks(X)$. Any non-singular hermitian form on a finitely-generated free $\Lambda$-module can be realized by one or two manifolds.
\end{theorem}

Let $X$ be a closed oriented $4$-manifold, $X$  is spin if its second Stiefel-Whitney class $w_2(X)=0$,  or equivalently $I(X, \ZZ_2)$ is even.
If  $H_1(X)$ has no 2-torsion, then  $I(X, \ZZ_2)$ is even if and only if  $I_X$ is even (recall Remark \ref{no 2-torsion}). If $\pi_1(X) \cong \ZZ_n$, when $n$ is odd, there are two $w_2$-types --- $X$ is either spin ($w_2(X)=0$) or non-spin ($w_2(X) \ne 0$);  when $n$ is even, there are three $w_2$-types: $X$ is of \emph{type I} (totally non-spin) if the universal cover $\widetilde X$ of $X$ is non-spin, of \emph{type II} if $X$ is spin, and of t\emph{ype III}  (almost spin) if $X$ is non-spin but $\widetilde X$ is spin. 

The classification of topological $4$-manifolds with finite cyclic fundamental group is due to  Hambleton and Kreck.  

\begin{theorem}\label{H-K}\cite[Theorem C]{HK2}. The homeomorphism class of a closed orientable 4-manifold $X$ with $\pi_1(X)\cong\mathbb{Z}_n$ is classified by the fundamental group, the intersection form $I_X$, the $w_2$-type, and the Kirby-Siebenmann invariant $ks(X)$.
\end{theorem}

For more general finite $\pi_1$,  Hambleton and Kreck  proved the following finiteness result.

\begin{theorem}\label {2} \cite{HK1}
There are only finitely many homeomorphism classes of closed orientable four-manifolds $Y$ with given Euler characteristic number $\chi(Y)$ and given finite fundamental group $\pi_1(Y)$.
\end{theorem}

\medskip

\noindent \textbf{Acknowledgement.} The authors would like to thank Ian Hambleton, Matthias Kreck and Hongbin Sun for helpful communications. The first author is supported by NSFC12071462. The second author is supported by NSFC11771021.

\section{Splitting of intersection forms and 1-dominations}

Suppose $X$ and $Y$ are closed orientable $4$-manifolds and $X$ 1-dominates $Y$. Then we have  
$I_X = I_Y \oplus L \ \ \ \text{and}\,\, I(X, \ZZ_2)=I(Y, \ZZ_2)\oplus L'\ \ (\star)$ 
by Lemma \ref{splitting}. 
In simply connected category, condition $(\star)$ is also a sufficient condition for 1-domination.

\begin{proposition}\label{DW}\cite{DW} 
Suppose $X$ and $Y$ are closed orientable simply connected four manifolds.
If $I_X=I_Y\oplus H$ (recall Remark \ref{no 2-torsion}), then $X$ 1-dominates  $Y$.
\end{proposition}

We wonder in what degree Proposition \ref{DW} can be extended from closed orientable 4-manifolds with trivial $\pi_1$ to closed orientable 4-manifolds with cyclic $\pi_1$. We will discuss this in next two subsections according to $\pi_1$ is infinite or finite.

\subsection{ $\pi_1=\ZZ$: 1-domination and stable 1-domination}

We will see  in this case $I_X=I_Y\oplus H$ does not guarantee that $X$ 1-dominates $Y$, but guarantee that
  $X$ stably 1-dominates $Y$. This fact is included in the following two Theorems \ref{stably 1-dom} and \ref{non 1-dom}.
  
  
\begin{theorem}\label{stably 1-dom} 
Suppose $X$ and $Y$ are closed orientable four manifolds, $\pi_1(X)=\pi_1(Y)=\mathbb{Z}$,
and $I_X=I_Y\oplus H$.
Then $X$ stably 1-dominates $Y$, precisely  there is a degree one map 
$$f: X\#(\#_3S^2\times S^2)\to Y\#(\#_3S^2\times S^2).$$
\end{theorem}

The proof uses the following decomposition result of Hambleton and Teichner.
\begin{theorem}\label{HT1}\cite{HT}
Let $X$ be a closed orientable four-manifold with $\pi_1(X)\cong\mathbb{Z}$. If $\beta_2(X)-|\sigma(X)|\ge 6$, then $X=M\# S^1\times S^3$ where $M$ is a closed simply connected 4-manifold.
\end{theorem}

\begin{proof}[Proof of Theorem \ref{stably 1-dom}]

Denote $$X^*=X\#(\#_3S^2\times S^2),\,\, Y^*= Y\#(\#_3S^2\times S^2). \qquad (2.0)$$

Suppose 
$$I_X=I_Y\oplus H.$$
From the connected sum construction of $X^*$ and $Y^*$, we still have
$$I_{X^*}=I_{Y^*}\oplus H.\qquad (2.1)$$

Since $\#_3S^2\times S^2$ is simply connected,  by Seifert-van Kampen's theorem we  have
 
 $$\pi_1(X^*)=\pi_1(Y^*)=\mathbb{Z}. \qquad (2.2)$$

Since
 $$\sigma(\#_3S^2\times S^2)=0,\,\, \text{and}\,\, \beta_2(\#_3S^2\times S^2)=6.$$
 
We have
$$\beta_2(X^*)=\beta_2(X)+\beta_2(\#_3S^2\times S^2)=\beta_2(X)+6.$$
and 
$$\sigma(X^*)=\sigma(X)+\sigma(\#_3S^2\times S^2)=\sigma(X).$$ 

Since $|\sigma(X)|\le \beta_2(X)$, then we have 
$$\beta_2(X^*)-|\sigma(X^*)|=\beta_2(X)+6-|\sigma(X)|\ge 6. \qquad (2.3)$$
By the same reason we have 
$$\beta_2(Y^*)-|\sigma(Y^*)|\ge 6.\qquad (2.4)$$

Now by (2.1), (2.3) and (2.4), we can apply Theorem \ref{HT1} to get

$$X^*=X_1\#S^1\times S^3, \,\, Y^*=Y_1\# S^1\times S^3\qquad (2.5)$$
for some simply connected 4-manifolds $X_1$ and $Y_1$.

Since $S^1\times S^3$ has vanishing intersection form, by comparing (2.2) and (2.5), we have 

$$I_{X_1}=I_{X^*}=I_{Y^*}\oplus H=I_{Y_1}\oplus H. \qquad (2.6)$$

Since both $X_1$ and $Y_1$ are simply connected, by (2.6) and Proposition \ref{DW},
there is a degree one map 
$$g: X_1\to Y_1.$$
Since $S^1\times S^3$ 1-dominates $S^1\times S^3$, by Lemma \ref{pinch} (2),
$X^*= X_1\#S^1\times S^3$ 1-dominates   $Y^*=Y_1\#S^1\times S^3$. Compare 
(2.0) and (2.5), we have a degree one map

$$f: X\#(\#_3S^2\times S^2)\to Y\#(\#_3S^2\times S^2).$$

This finishes the proof.
\end{proof}

\begin{theorem}\label{non 1-dom}
There exist two closed orientable 4-manifolds $X$ and $Y$ with 

(i) $\pi_1(X)\cong\pi_1(Y)\cong \mathbb Z$ and 

(ii) $I_X=I_Y\oplus H$, 

(iii) $X$ does not 1-dominates $Y$.
\end{theorem}

We need some preparation to prove Theorem \ref{non 1-dom}.

Let $\iota: \mathbb{Z}\rightarrow \Lambda$, $1 \mapsto 1$ be the canonical inclusion and treat $\mathbb{Z}$ as a subring of $\Lambda$. 
Let $N$ be a finitely generated 
free $\Lambda$-module,
$$\langle \  , \ \rangle \colon N \times N \to \Lambda$$
be a hermitian form on $N$. We say $\langle \  , \ \rangle$ is extended from $\mathbb{Z}$ if there exists a
basis $\{b_i\}$ of $N$ such that $\langle b_i, b_j \rangle \in\mathbb{Z}\subset \Lambda, \,\, \forall i,j.$

Let $s=t+t^{-1}$. Let
$$A=\left(
\begin{matrix}
1+s+s^2 & s+s^2 & 1+s & s\\
s+s^2 & 1+s+s^2 & s & 1+s\\
1+s & s & 2 & 0\\
s & 1+s & 0 & 2
\end{matrix}
\right).$$
then $A$ describes a hermitian form of rank $4$ over $\Lambda$. 

The proof of Theorem \ref{non 1-dom} needs the following algebraic result on hermitian forms over $\Lambda$ due to Hambleton and Teichner \cite{HT}.
\begin{theorem}[HT, Thm 1.1] \label{HT2}
The Hermitian form $A$ is not extended from $\mathbb{Z}$.
\end{theorem}

It is also easy to see the following

\begin{proposition}
If $X=S^1\times S^3\# X^\Delta$ for some simply connected 4-manifold $X^\Delta$, then $S(\widetilde X)$ is extended from $\mathbb{Z}$.
\end{proposition}

\begin{proof}[Proof of Theorem \ref{non 1-dom}]
Let $A$
be the 4$\times$4 matrix over the ring $\Lambda$ as above and $N$ be the Hermitian form determined by $A$. 
 By Theorem \ref{HT2},
 $N$ is not extended from $\mathbb{Z}$, i.~e., there does not exist a bilinear form $N_1$ over $\mathbb{Z}$ such that $N\cong N_1\otimes_{\mathbb{Z}}\Lambda$. 
 
Applying \cite{FQ}, there exists a closed orientable 4-manifold $Y$ with $\pi_1(Y) \cong \mathbb Z$ whose equivariant intersection form $S(\widetilde Y)$ on $\pi_2(Y)$ is isomorphic to $N$. 
By the work of Freedman, there exists a simply connected closed 4-manifold $X_1$ such that its intersection form $I_{X_1}\cong I_Y$. Let $$X=X_1\#S^1\times S^3.$$ Then 
$$\pi_1(X)=\pi_1(X_1) * \pi_1(S^1\times S^3)=\pi_1(S^1 \times S^3)=\mathbb{Z}$$
and
$$I_{X}=I_{X_1}\oplus
I_{S^1\times S^3}=I_{X_1}\cong I_{Y}.$$
So the conditions (i) and (ii) are satisfied. Next we prove (iii). 

If not, there exists a degree 1 map $f:X\rightarrow Y$. Then $f$ induces a surjective map
$$f_{\pi_1}:\pi_1(X)\cong\mathbb{Z}\rightarrow\pi_1(Y)\cong\mathbb{Z}.$$
Since any surjection on $\mathbb{Z}$ is an injection, $f_{\pi_1}$ is an isomorphism. It's known that $\pi_2(X)$ is a  free $\Lambda$-module of rank $\beta_2(X)$ \cite[\S 3]{H}, i.~e., $\pi_2(X) \cong\Lambda^{\beta_2(X)}$ as an 
$\Lambda$-module, and so is $\pi_2(Y)\cong\Lambda^{\beta_2(Y)}$.
Since $I_X\cong I_Y$, we have $\beta_2(X)=\beta_2(Y)$. Since $f$ is a degree 1 map, it induces a surjective $\Lambda$-module homomorphism
$$f_{\pi_2}:\pi_2(X)\rightarrow \pi_2(Y)$$
preserving the Hermitian form. 
Since 
$$\pi_2(X)\cong\Lambda^{\beta_2(X)}\cong\Lambda^{\beta_2(Y)}\cong\pi_2(Y)$$ 
are finitely generated free $\Lambda$ modules of the same rank, $f_{\pi_2}$ is an isomorphism.
So $f_{\pi_2}$ is an isomorphism between the hermitian forms $S(\widetilde X)$ and $S(\widetilde Y)$. But
since $X=X_1\#S^1\times S^3$, the equivariant intersection form $S(\widetilde X)$ is extended from $\mathbb Z$, contradicting the fact that $N$ is not extended, thus  (iii) is proved. 
\end{proof}

\subsection{$\pi_1=\ZZ_n$: Connected sum decomposition and 1-domination}

Let $f \colon X \to Y$ be a degree $1$ map, then we have the condition  
$$I_X = I_Y \oplus L \ \ \text{and} \ \ I(X, \ZZ_2)=I(Y, \ZZ_2)\oplus L'\ \ ( \star).$$

Recall Remark \ref{no 2-torsion},  when $n$ is odd,  $I(X, \ZZ_2)=I(Y, \ZZ_2)\oplus L'$ follows from I$_X = I_Y \oplus L$, but for $n$ even this is an independent condition.

\begin{theorem}\label{1-dom n}
Suppose $X$, $Y$ are closed oriented  $4$-manifolds with $\pi_1(X)=\pi_1(Y)=\ZZ_n$. 
Then the condition $(\star)$  are sufficient conditions for the existence of a degree one map $ X \to Y$ unless $X$ is of type III and $Y$ is of type II for even $n$.
\end{theorem}

The proof of Theorem \ref{1-dom n} is based on Theorem \ref{decomposition} which provides the decomposition of the manifolds $X$ with $\pi_1=\ZZ_n$ into a 
connected sum of a rational homology 4-sphere with $\pi_1=\ZZ_n$ and a simply connected 4-manifold. 
The following several lemmas are used in the proof of Theorem \ref{decomposition}.



There is an alternative description of the $w_2$-types. 

\begin{lemma}\label{Su} Suppose $X$ is a closed orientable  4-manifold
 with $\pi_1(X)=\ZZ_n$ and $n$ even. Then  $I_X$ is an even form if and only if $X$ is of type II or III.
\end{lemma}

\begin{proof}
By Wu formula $I_X$ is even if and only if the evaluation of $w_2(X)$ on $H_2(X)$ is trivial. By \cite[Lemma 2.3]{HS}
$X$ is of type III if and only if $w_2(X) \ne 0$ and its image in $\mathrm{Hom}(H_2(X), \ZZ_2)$ is $0$.
\end{proof}

\begin{lemma}\label{FK} Suppose $X$ is a closed orientable  4-manifold.
Then $X$ is spin implies that $ks(X)=\sigma(X)/8 \pmod 2$.
\end{lemma}

\begin{proof} By Freedman-Kirby formula \cite{FK78}.
\end{proof}

\begin{lemma}\label{spin}
Suppose there is a degree $1$ map $X \to Y$.

(1) If  $X$ is spin, then $Y$ is also spin.

Furthermore if $\pi_1(X)=\pi_1(Y)$, then

(2) If $X$ is almost spin, then $Y$ is almost spin or spin.
\end{lemma}
\begin{proof}
Since $I(Y;\mathbb Z_2)$ is a direct summand of $I(X;\mathbb Z_2)$, the characteristic element $w_2(X)$ of $I(X;\mathbb Z_2)$ is also a characteristic element of $I(Y;\mathbb Z_2)$. By the uniqueness of characteristic elements $w_2(X)=0$ implies $w_2(Y)=0$. If $\pi_1(X)=\pi_1(Y)$, then $f$ induces an isormorphism on $\pi_1$ hence lifts to a degree one map $\widetilde f$ between the universal covers $\widetilde f \colon \widetilde X \to \widetilde Y$. Then $\widetilde X$ is spin implie $\widetilde Y$ is spin.
\end{proof}

A rational homology 4-sphere $\Sigma$ has vanishing integral intersection form $I_\Sigma$. It is known that  finite cyclic group can be realized as the $\pi_1$ of  rational homology 4-spheres. 
The next proposition will provide a classification of such realizations.

\begin{proposition}\label{RHS}
Suppose $\Sigma$ is a rational homology 4-sphere with  $\pi_1(\Sigma)=\ZZ_n$. 

(i) If $n$ is odd then $\Sigma$ is spin; if $n$ is even then $\Sigma$ is of $w_2$-type II or III.

(ii) For $n$ odd or $n$ even and $w_2$-type II, the homeomorphism type of $\Sigma$ is unique.

(iii) For $n$ even and $w_2$-type III, there are two homeomorphism types of $\Sigma$, which are homotopy equivalent.
\end{proposition}

\begin{proof} Note first as a rational homology 4-sphere, $\beta_2(\Sigma)=0$ and $I_\Sigma$ is trivial, and $\sigma(\Sigma)=0$.

(i) If $n$ is odd, then $w_2(\Sigma)\ne 0$ implies $\beta_2(\Sigma)>0$.
If $n$ is even, then  $\Sigma$ is of type I still implies $\beta_2(\Sigma)>0$ by Lemma \ref{Su}.
Both cases are ruled out by that $\beta_2(\Sigma)=0$.
 
(ii) If $n$ is odd and $w_2=0$,  then $\Sigma$ is spin. Then by $\sigma(\Sigma)=0$ and Lemma \ref{FK}, $ks(\Sigma)=0$.
So $\Sigma$ is unique by Theorem \ref{H-K}, which will be denoted as $\Sigma_*$.
In the case of type II, by the same reason, $\Sigma$ is unique,  which will be denoted as $\Sigma_0$.

(iii) The existence of two homeomorphism types in $w_2$-type III is claimed in \cite[Prosition 4.1]{HK2}. Denote the manifolds by $\Sigma_{1,i}$ ($i=0,1$), with Kirby-Siebenmann invariant $ks(\Sigma_{1,i})=i$. The following argument (due to Ian Hambleton) shows they are homotopy equivalent. 
Take $x \in H^2(\Sigma_{1,0};\ZZ_2)$ with $x^2 \ne 0$ (such an $x$ exists since we are in $w_2$-type III). The set of normal invariants $\mathcal N^{\mathrm{top}}(\Sigma_{1,0})$ is identified with $H^2(\Sigma_{1,0};\mathbb Z_2) \oplus H^4(\Sigma_{1,0};\mathbb Z)$ (c.~f.~\cite{Wall}) The surgery obstruction for the surgery problem corresponding to $x$ is trivial (\cite[Proposition 7.4]{HMTW}), hence we get a closed manifold homotopy equivalent to $\Sigma_{1,0}$ with different Kirby-Siebenmann invariant (c.~f.~\cite[p.~398]{KT}), which is $\Sigma_{1,1}$.  
\end{proof}

Now we have a decomposition theorem for oriented $4$-manifolds with finite cyclic fundamental group. As a consequence of Proposition \ref{RHS} and Theorem \ref{H-K} this decomposition result could have independent interest. 

\begin{theorem}\label{decomposition}
Suppose $X$ is a closed orientable 4-manifold and $\pi_1(X)=\ZZ_n$. Then 
$X=\Sigma\# M,$  where $\Sigma$ is a rational homology sphere and $M$ is a  simply-connected closed 4-manifold, and $I_M=I_X$.
Precisely (following the notions in the proof of Proposition \ref{RHS}):

(i) $X=\Sigma_*\# M$ if $n$ is odd;

(ii) $X=\Sigma_0\# M$ if $X$ is of type II;

(iii) $X=\Sigma_{1, i}\# M$ if $X$ is of type III;

(iv) $X \cong \Sigma_0 \# M \cong \Sigma_{1,0} \# M \cong \Sigma_{1,1} \sharp M^*$ if $X$ is of type I,
where $M$ and $M^*$ have different Kirby-Siebenmann invariants.

In cases (i), (ii) and (iii), the decomposition is unique.
\end{theorem}

\begin{proof}
Let $M$ be a simply-connected closed $4$-manifold with intersection form $I_M=I_X$. Then $\sigma(M)=\sigma(X)$.

(i) Suppose first  $n$ is odd.

If  $X$ is  spin,  then  $I_X$ is an even form, therefore $M$ is a spin manifold. Since both $X$ and $M$ are spin and $\sigma(M)=\sigma(X)$,
we have $ks(M)=ks(X)$ by Lemma \ref{FK}.

If  $X$ is  non-spin, then  $I_X$ is an odd form hence $M$ is non-spin,  we further require that the Kirby-Siebenmann invariant of $M$ equals that of $X$. 
Then $X$ and $\Sigma_* \sharp M$ have the same intersection form, $w_2$-type and Kirby-Siebenmann invariant. Therefore we have a homeomorphism $X \cong \Sigma_* \sharp M$, where $M$ is uniquely determined by $X$. 

 Suppose then  $n$ is even. 

(ii) $X$ is of type II: By the same arguments as above, there is a homeomorphism $X \cong \Sigma_0 \sharp M$ by Theorem \ref{H-K}, where $M$ is uniquely determined by $X$.

(iii) $X$ is of type III:  By Lemma \ref{Su} $I_X$ is even, then $M$ is still spin. 
Since $\Sigma_{1,i}$ is of type III,  it is easy to see that $\Sigma_{1,i} \sharp M$ is still of type III.
By Lemma \ref{FK}, $ks(M)=\sigma(M)/8 \pmod 2$.
 Take $i =ks(X)-\mathrm{sign}(X)/8 \pmod 2$, 
then $X$ and  $\Sigma_{1,i} \sharp M$ have equal Kirby-Siebenmann invariant. Therefore  $X \cong \Sigma_{1,i} \sharp M$ by Theorem \ref{H-K},
where $M$ is uniquely determined by $X$.

(iv) $X$ is of type I: Let $M$ be the simply-connected $4$-manifold with $I_M = I_X$ and $ks(M)=ks(X)$, $M^*$ be the simply-connected $4$-manifold with $I_{M^*} = I_X$ and $ks(M^*)=1-ks(X) \pmod 2$. 
By Lemma \ref{Su}, $I_X$ is odd form. So both $I_M$ and $I_{M^*}$ are odd form, hence both $M$ and $M^*$ are type I by  Lemma \ref{Su}.
Then it is easy to see those 4-manifolds $X$, $\Sigma_0 \sharp M$,  $\Sigma_{1,0} \sharp M$, $\Sigma_{1,1} \sharp M^*$ have the type I,
 the same intersection forms and the same $ks$ invariant. 
Then there are homeomorphisms 
$$X \cong \Sigma_0 \sharp M \cong \Sigma_{1,0} \sharp M \cong \Sigma_{1,1} \sharp M^*$$
by Theorem \ref{H-K}.
\end{proof}

\begin{proof}[Proof of Theorem \ref{1-dom n}] Suppose $X$, $Y$ are as stated in Theorem \ref{1-dom n}.

By Theorem \ref{decomposition}, we have decompositions $X \cong \Sigma \sharp M$ and $Y \cong \Sigma' \sharp N$,
 where $\Sigma$ and $\Sigma'$ are appropriate rational homology spheres,  $M$ and  $N$ are closed simply connected 4-manifolds 
 with $I_M=I_X$ and $I_N=I_Y$. The condition ($\star$) implies that  $I_N$ is an orthogonal summand of $I_M$, therefore there is a degree $1$ map $M \to N$. 
 
(a) If $n$ is odd, then $\Sigma=\Sigma'=\Sigma_*$ by Theorem \ref{decomposition} (i).
 
 Below we suppose that $n$ is even.  The following cases are ruled out by Lemma \ref{spin}: $X$ is type II and $Y$ is either type I
 or III; and $X$ is type III and $Y$ is type I. So we have only the following several cases to discuss:
 
(b) Both $X$ and $Y$ are of  type II, then $\Sigma=\Sigma'=\Sigma_0$ by Theorem \ref{decomposition} (ii).  

(c)  Both $X$ and $Y$ are of type III, then $\Sigma \simeq \Sigma'$ by Theorem \ref{decomposition} (iii) and Proposition \ref{RHS} (iii).

(d) $X$  is of type I. Then  $X$  has three decompositions by Theorem \ref{decomposition} (iv),
and we may choose $\Sigma$ so that $\Sigma =\Sigma'$. 
  
In these cases we get a degree $1$ map $\Sigma \to \Sigma'$. Note also $M$ and $M^*$ are homotopy equivalent. So get a degree 1 map $X \to Y$ by Lemma \ref{pinch} (ii). \end{proof}

\begin{remark} The undecided  case in Theorem \ref{1-dom n}  is that $X$ is of type III and $Y$ is of type II.  We post
a concrete candidate to address. 

Let $\Sigma^2_{1,0}$ and $\Sigma^2_0$ be the rational homology spheres $\Sigma_{1,0}$ and $\Sigma_0$ with $\pi_1=\ZZ_2$ respectively. According to \cite[Remark 4.5]{HK1}: $$\Sigma^2_0=S(\eta \oplus \varepsilon^2)\ \ \mathrm{and} \ \ \Sigma^2_{1,0}=S(3\eta),$$
where $\eta$ is the non-trivial real line bundle over $\mathbb R \mathrm P^2$ and  $S(\eta \oplus \varepsilon^2)$ and  $S(3\eta)$ are   the sphere bundles of the vectors bundles $\eta \oplus \varepsilon^2$ and $3 \eta$ respectively. 
Let $X=\Sigma^2_{1,0} \# S^2\times S^2$ and $Y=\Sigma^2_0$. Then it is easy to see $X$ is of type III and $Y$ is of type
II. Indeed we also have $$I(X, \ZZ_2)=I(Y, \ZZ_2)\oplus H\qquad (*)$$
\begin{question}
Does $\Sigma^2_{1,0} \# S^2\times S^2$ 1-dominates $\Sigma^2_0$?
\end{question}

%

\end{remark}

\section{On finiteness of 1-domination when target manifolds with $\pi_1=\ZZ$ (resp. $\ZZ_n$)}

It is known that each closed orientable 3-manifold 1-dominates finitely many closed orientable 3-manifolds \cite{Liu}. 
We wonder  when the above result still hold in dimension 4.

\begin{theorem}\label{infinite cyclic}
For each positive integer $n$, there exists a closed orientable 4-manifold $Y$ which 1-dominates all closed orientable 4-manifolds $X$ with $\pi_1(X)=\ZZ$ and $\beta_2(X) \le n$. 
\end{theorem}

The proof of Theorem \ref{infinite cyclic} needs Theorem \ref{HT1} and the following result follows from Freedman's work and a result in \cite{MH}.

\begin{proposition}\label{Free}
Given a positive integer $n$, there exist only finitely many closed simply-connected 4-manifolds with second betti number not greater than $n$.
\end{proposition}

\begin{proof}[Proof of Theorem \ref{infinite cyclic}] Fix a positive integer $n$.
Let $Y_1$ be the connected sum of all closed simply connected 4-manifolds with the second Betti number $\beta_2\le n+6$. \
By Proposition \ref{Free}, $Y_1$ is a closed orientable 4-manifold.

Then by Lemma \ref{pinch}, $Y_1$ one dominates every closed simply connected 4-manifold $Z$ with $\beta_2(Z)\le n+6$.
Let 
$$Y=Y_1\# S^1\times S^3.$$

Let $X$ be any closed orientable 4-manifold with $\pi_1(X)\cong \mathbb{Z}$ and $\beta_2(X)\le n$.
We will prove that $Y$ 1-dominates $X$. 
  Let $$X^*=X\#(\#_3S^2\times S^2).$$
  As we see in the proof of Theorem \ref{stably 1-dom}
  $$\pi_1(X^*)=\pi_1(X)*\pi_1(\#_3S^2\times S^2)\cong\mathbb{Z}.$$

and 
$$\beta_2(X^*)-|\sigma(X^*)|\ge 6.$$
So by Theorem \ref{HT1}, we have 
$$X^*\cong M\# S^1\times S^3$$ for some simply connected 4-manifold $M$.

Since 
$$\beta_2(X)=\beta_2(M)+\beta_2(S^1\times S^3) \,\, \text {and}\,\, \beta_2(S^1\times S^3)=0,$$ 
we obtain $$\beta_2(M)=\beta_2(X_1)=\beta_2(X)+6.$$
Since $$\beta_2(X)\le n,$$
we have 
$$\beta_2(M)\le n+6.$$ 

By the construction of $Y_1$, $M$ is a connected sum factor of $Y_1$. So
 $$Y_1\cong M\#N$$ for some 4-manifold $N$.
Then 
$$Y\cong Y_1\# S^1\times S^3$$
$$\cong M\#N\# S^1\times S^3$$
$$\cong (M\# S^1\times S^3)\#N $$
$$\cong X^*\#N $$$$\cong X\#(\#_3S^2\times S^2)\#N.$$
So $Y$ one dominates $X$.
\end{proof}

\begin{corollary}\label{3conditions}
The following statements are equivalent:

{\rm (i)} For each positive integer $n$, there exist only finitely many equivalence classes of non-singular hermitian forms over free modules over $\mathbb Z[\mathbb Z]$ of rank at most $n$.

{\rm (ii)} For each positive integer $n$, there are only finitely many closed 4-manifolds $X$ with $\pi_1(X)\cong\mathbb{Z}$ and $\beta_2(X)\le n$.

{\rm (iii)} Any  closed orientable 4-manifold $Y$ 1-dominates  only finitely many homeomorphism classes of closed orientable 4-manifolds with infinite cyclic  fundamental group.

\end{corollary}

\begin{proof}
The equivalence of (i) and (ii) follows from Theorem \ref{F-Q}. 

The equivalence of (ii) and (iii) follows from Theorem \ref{infinite cyclic}. 


\end{proof}

Since so far we do not know if Corollary \ref{3conditions} (i) holds, we do not know if each closed orientable 4-manifold 1-dominates only finitely many closed orientable 4-manifolds with infinite cyclic $\pi_1$. However we can prove

\begin{theorem}\label{cyclic-stable}
Every closed orientable 4-manifold $X$  1-dominates at most finitely many stable homeomorphism classes of closed orientable 4-manifolds $Y$ with $\pi_1(Y)=\mathbb{Z}$.
\end{theorem}
\begin{proof}
Suppose $X$ 1-dominates $Y$.  Then we know that
$$\beta_2(X) \ge \beta_2(Y).$$ 
Let 
$$Y^*=Y\#(\#_3S^2\times S^2).\qquad (3.1)$$
 Then as we see in the proof of Theorem \ref{stably 1-dom}
 $$\beta_2(Y^*)-|\sigma(Y^*)|\ge 6.$$
Then by Proposition \ref{HT1},  

$$Y^*\cong M\# S^1\times S^3 \qquad (3.2)$$ with $M$ a closed simply connected 4-manifold. 
Since $$\beta_2(Y^*)=\beta_2(M)+\beta_2(S^1\times S^3)=\beta_2(M),$$
 we get $$\beta_2(M)=\beta_2(Y)+6\le\beta_2(X)+6.$$
Since there are at most finitely many closed simply connected 4-manifold 
with second Betti number not greater than $\beta_2(X)+6$ by Proposition \ref{Free}, 
there are at most finitely many closed orientable  4-manifolds of the form $$M\#S^1\times S^3$$ 
with $M$ simply connected. Comparing (3.1) and (3.2), 
there are at most finitely many closed orientable  4-manifolds of the form
$$Y^*=Y\#(\#_3S^2\times S^2)$$
So $X$ 1-dominates at most finitely many stable homeomorphism classes of closed orientable 4-manifolds $Y$ with $\pi_1(Y)=\mathbb{Z}$.
\end{proof}

Even we do not know if each closed orientable 4-manifold 1-dominates only finitely many closed orientable 4-manifolds with $\pi_1=\ZZ$, it is easy to derive that each closed orientable 4-manifold 1-dominates only finitely many closed orientable 4-manifolds with $\pi_1=\ZZ_n$
for all positive integer $n$. Indeed we have

\begin{theorem}\label{finite-abelian}
Every closed orientable 4-manifold $X$ 1-dominates at most finitely many closed orientable 4-manifolds $Y$ with finite abelian $\pi_1$.
\end{theorem}

\begin{proof}
Suppose $f: X\to Y$ is a map of degree one between closed orientable 4-manifolds. By Lemma \ref{splitting} (2) we have 
$$H_*(X)=H_*(Y)\oplus A.$$
In particular
 $$\beta_2(X)\ge \beta_2(Y)$$
 and 
$$\text{Tor}H_1(X)=\text{Tor}H_1(Y)\oplus A'.$$
Since $\pi_1(Y)$ is finite and abelian, we have 
$$\pi_1(Y)=H_1(Y)=\text{Tor}H_1(Y).$$
In particular $\pi_1(Y)$ is a subgroup of the finite abelian group $\text{Tor}H_1(X)$, it has only finitely many choices.

By the Hurewicz's theorem, $H_1(Y,\mathbb{Z})$ is isomorphic to the abelianization of $\pi_1(Y)$, and is hence finite.
So $\beta_1(Y)=0$. By the Poincare duality and  the Euler-Poincare formula (Lemma \ref{Euler}), we have 
$$\chi(Y)=2+\beta_2(Y).$$
We get $$2\le \chi(Y)\le 2+\beta_2(X).$$

So the pair $(\pi_1(Y), \chi(Y))$ has only finitely many choices.
For each choice of $(\pi_1(Y), \chi(Y))$,
by Theorem \ref{2},  there are only finitely many closed orientable 4-manifolds $Y$ realizing $(\pi_1(Y), \chi(Y))$. So there are only finitely many  closed orientable 4-manifolds $Y$ 1-dominated by $X$.
\end{proof}

\section{$1$-domination to manifolds with minimal Euler characteristic}

The 4-manifold Euler characteristic  for a finitely presented group $G$ is  defined by  
$$\chi_4(G)=\text{inf}\{\chi(X)| \text{$X$ is a closed orientable 4-manifold and $\pi_1(X)\cong G$}\},$$ 
and we have  

\begin{theorem}\label{1} \cite{Ko}, \cite{Hi1}, \cite{KL}, \cite{SunW1}. For a closed and orientable 3-manifold $Q$, we have 
$$\chi_4(\pi_1(Q))=2-2q,$$ 
where $q$ is the maximal rank of free group in the free product decomposition of $\pi_1(Q)$.
\end{theorem}

The following question is raised in \cite{SunW2}:

\begin{question}\label{q43} Suppose $Q$ is a closed orientable 3-manifold and $X$  is a closed orientable 4-manifold 
with $\pi_1(X)\cong \pi_1(Q)$. Is there a degree-$1$ map $f:X\to N$ for a closed orientable 4-manifold $N$ realizing $\chi_4(\pi_1(Q))$? 
\end{question}

Question \ref{q43} has a positive answer when $G$ is cyclic. 

\begin{theorem}\label{answer}
Suppose $G$ is a cyclic group and $X$ is a closed orientable 4-manifold with $\pi_1(X)\cong G$. Then $X$ $1$-dominates a 4-manifold $N$ realizing $\chi_4(G)$. 
\end{theorem}

\begin{proof}
(i)  Suppose $G=\ZZ$.   Since $\ZZ$ is the fundamental groups of closed orientable 3-manifolds $S^1\times S^2$, we have  $\chi_4(\ZZ)=0$ by Theorem \ref{1}. 

Suppose $N$ is a closed orientable 4-manifold realizing $\chi_4(\ZZ)$. Then we have  $\chi(N)=0$. Note  $\beta_1(N)=1$.
By the Poincare duality and the Euler-Poincare formula (Lemma \ref{Euler}), $$0=\chi(N)=2-\beta_1(N)+\beta_2(N)=\beta_2(X).$$
Hence $\beta_2(N)=0$. Since $\pi_1(N)=\ZZ$,  $N$ is homeomorphic to $S^1\times S^3$ by Theorem 2.5. 
Then Theorem \ref{answer} for $G=\ZZ$ follows from the following

\begin{proposition}\label{Q^*}
Suppose $M$ is a closed orientable 4-manifold with $\pi_1(M)=\ZZ$. Then  $M$ 1-dominates $S^1\times S^3$.
\end{proposition}
\begin{proof}
Let $$M_1=M\#(\#_3S^2\times S^2)\qquad(4.1).$$ 
As we see in Theorem \ref{stably 1-dom}
$$\pi_1(M_1)=\pi_1(M)*\pi_1(\#_3S^2\times S^2)\cong \mathbb{Z},$$
and
$$\beta_2(M_1)-|\sigma(M_1)|\ge 6.$$
By Proposition \ref{HT1},
$$M_1\cong S^1\times S^3\#M^\Delta\qquad(4.2)$$ for some closed simply connected 4-manifold $M^\Delta$. 
So there is a degree 1 map $$g:M_1\rightarrow S^1\times S^3.$$  by Lemma \ref{pinch}. 
Since $\pi_2(S^1\times S^3)=\pi_2(S^1)\times \pi_2(S^3)=0$, $g$ induces a trivial homomorphism on $\pi_2$: 
$$g_{\pi_2}:\pi_2(M_1)\rightarrow \pi_2(S^1\times S^3).$$
Let $\Sigma_1,\Sigma_1^*,\Sigma_2,\Sigma_2^*,\Sigma_3,\Sigma_3^*$
be six 2-spheres which are $S^2\times *$ and $*\times S^2$ in $\#_3S^2\times S^2$ in the connected sum decomposition (4.1). 
Attaching to each 2-sphere above a 3-cell via the attaching map
$$\phi_i:\partial D_i^3\rightarrow \Sigma_i\,\,\, \text{or}\,\,\, \phi_i^*:\partial D_i^{3*}\rightarrow \Sigma_i^*,$$
where each $\phi_i$ or $\phi_i^*$ is a homeomorphism, 
we obtain a CW-complex $M_2$. 
Note the space obtained by attaching two $3$-cells to $S^2\times S^2$  is homotopy equivalent  to $S^4$. 
It follows that $M_2$ is homotopy equivalent  to the connected sum of $M$ and three $S^4$, therefore  is homotopy equivalent to $M$.
 So there are two maps 
 $$h: M_2 \rightarrow M\,\,\, \text{ and}\,\,\,  h^*: M \rightarrow M_2$$ such that 
$h^*\circ h$ is homotopic to the identity of $M_2$ and $h\circ h^*$ is homotopic to the identity of $M$.  Let   $$\iota: M_1\to M_2$$ be the inclusion map. Since $g_{\pi_2}=0$,  
we can extend $g$ to 3-cells $D_i^3$ and $D_i^{3*}$, and then we get a map 
$$g_1: M_2\rightarrow S^1\times S^3$$
with
$g=g_1\circ\iota.$ Let 
$$g_2=h^*\circ\iota:M_1\to M_2\rightarrow M\,\,\, \text{and}\,\,\,  g_3=g_1\circ h:M\rightarrow  M _2 \to S^1\times S^3$$ be two composite maps. 
Then 
$$g_3\circ g_2=g_1\circ h\circ h^*\circ \iota: M_1\to M \to S^1\times S^3$$ is homotopic to $g=g_2\circ \iota$, so is a degree 1 map. It follows that $g_3$ is a degree 1 map from $M$ to $S^1\times S^3$. We finish the proof.
\end{proof}

(ii)  Suppose $G=\ZZ_n$.   Since $\ZZ_n$ is the fundamental groups of closed orientable 3-manifold, namely the lens space $L(n,m)$, we have  $\chi_4(\ZZ_n)=2$ by Theorem \ref{1}. 

Suppose $X$ is a closed orientable 4-manifold and $\pi_1(X)=\ZZ_n$. Then 
$X=\Sigma\# M$ by Theorem \ref{decomposition}, where $\Sigma$ a rational homology sphere and $M$ is   simply-connected.
So there is a degree one map $X\to \Sigma$. Since $\Sigma$ is a rational  homology sphere,
we have $\beta_1(\Sigma)=\beta_2(\Sigma)=\beta_3(\Sigma)=0$. Then by Euler-Poincare formula (Lamma \ref{Euler})
$$\chi(\Sigma)=2.$$
Hence $\Sigma$ realizes $\chi_4(\ZZ_n)=2$. 
This finishes the proof of Theorem \ref{answer}  for $G=\ZZ_n$
\end{proof}

\begin{remark}
Indeed Theorem \ref{decomposition} itself implies $\chi_4(\ZZ_n)=2$:  Suppose $X$ is a closed orientable 4-manifold and $\pi_1(X)=\ZZ_n$. Then 
$X=\Sigma\# M,$  where $\Sigma$ a rational homology sphere and $M$ is   simply-connected by Theorem \ref{decomposition}.
As we see that $\chi(\Sigma)=2$. Since $M$ is simply connected, we have $\beta_1(M)=\beta_3(M)=0$, then 
$\chi(M)=2+\beta_2(M)\ge 2$. So
$$\chi(X)=\chi(\Sigma\# M)= \chi(\Sigma)+\chi( M)-2=\chi( M)\ge 2.$$
That is to say $\chi_4(\ZZ_n)=2$.
\end{remark}

\end{document}